\documentclass[12pt]{article}

\usepackage[a4paper]{geometry} 
\usepackage{latexsym}         
\usepackage{amsmath,amsthm}          
\usepackage[applemac]{inputenc} 
\usepackage{amsfonts,amssymb} 
\usepackage{tensor} 
\usepackage{enumerate}
\usepackage[english]{babel}
\ifx\undefined\og\relax%
\newcommand{\og}{\guillemotleft} 
\fi

\newtheorem{theorem}{Theorem}[section]
\newtheorem{lemma}[theorem]{Lemma}
\newtheorem{proposition}[theorem]{Proposition}

\theoremstyle{definition}  
\newtheorem{definition}[theorem]{Definition}

\newtheorem*{Thanks}{Thanks}

\newlength{\blackboxsize} \blackboxsize=1.2ex

\newenvironment{permatrix}{%
\left(\begin{array}{cccc|cccc}%
}{%
\end{array}\right)%
}

\newcommand{\T}{\textup{T}} 
\newcommand{\NormBundle}{\mathcal N} \newcommand{\RR}{\mathbb{R}}
\newcommand*{\var}[1]{\mathbf{#1}} 

\newcommand{\lsemiprod}{\ltimes} 
\newcommand{\lsemisum}{\ltimes} %
\newcommand{\End}{\operatorname{End}}
\renewcommand{\Im}{\operatorname{Im}}
\newcommand{\rk}{\operatorname{rk}}
\newcommand{\Tr}{\operatorname{Tr}}
\newcommand{\Ker}{\operatorname{Ker}}
\newcommand{\ad}{\operatorname{ad}}
\newcommand{\Id}{\operatorname{Id}}
\newcommand{\cdv}{\mathfrak{X}} 
\newcommand*{\name}[1]{\textup{#1}}
\newcommand*{\Defn}[1]{\emph{#1}} 
\newcommand*{\alg}[1]{\mathcal{#1}} 

\newcommand{\telque}{\,\vert\,} 

\newcommand{\pardef}{:=}
\newcommand{\notation}{\stackrel{\textup{not}}=}
\newcommand{\Sp}{\operatorname{Sp}}

\renewcommand{\sp}{\operatorname{sp}}
\newcommand{\affinealgebra}{\alg A}
\newcommand{\affinegroup}{\var A}

\let\Exp\exp 
\renewcommand*{\exp}[1]{e^{#1}} 
\renewcommand{\tensor}{\otimes} %
\newcommand*{\transpose}[1]{{\vphantom{#1}}^{\mathit
    t}{#1}} 
\renewcommand{\flat}{\mathring}
\newcommand{\NhOmega}{\Omega^{\mathcal{N}_0}}
\newcommand{\NbOmega}{\Omega_{\mathcal{N}_0}}
\newcommand{\tnabla}{\widetilde\nabla}
\newcommand{\basep}{{m_o}} 
\newcommand{\basepzero}{{0}}
\newcommand{\ddtzero}{\frac{\text{d}}{\text{d}t}_{\vert0}} 
\newcommand\rond{\circ}

\newcommand*{\csum}[3]{\cyclic_{#1,#2,#3}} 
\newcommand*{\cyclic}{\mathop{\kern0.9ex{{+}\kern-2.2ex\raise-.29ex%
      \hbox{\Large\hbox{$\circlearrowright$}}}}\limits}

\newcommand{\inner}{\imath} 
\newenvironment{bigproof}{
 \textbf{Proof}\hspace{.7cm}%
 \setcounter{equation}{0}
 \renewcommand\theequation{\alph{equation}}
}{}

\title{Extrinsic symplectic symmetric spaces}

\author{
  Michel Cahen$^{*,1}$,  Simone Gutt$^{*,1,2}$\\
  Nicolas Richard$^1$ and Lorenz Schwachh\"ofer$^3$\\
  \texttt{\small mcahen@ulb.ac.be, sgutt@ulb.ac.be}\\
  \texttt{\small nrichard@ulb.ac.be, Lorenz.Schwachhoefer@math.uni-dortmund.de}\\
   \textit{\small $*$ Acad\'emie Royale de Belgique} \\
\textit{\small $1$ Universit\'e Libre de Bruxelles, Campus Plaine CP 218,} \\
  \textit{\small Bvd du Triomphe, B-1050~Brussels, Belgium}\\
  \textit{\small  $2$ Universit\'e Paul Verlaine - Metz, LMAM}\\
  \textit{\small Ile du Saulcy, F-57045~Metz Cedex 01, France}\\
  \textit{\small $3$ Universit\"at Dortmund, Fachbereich Mathematik,
    Lehrstuhl LSVII}\\
  \textit{\small Vogelpothsweg 87, 44221 Dortmund, Germany}}%
\date{}

\numberwithin{equation}{section} 

\begin{document}
\maketitle
\setcounter{page}{0}
\thispagestyle{empty}

\begin{abstract}
  We define the notion of extrinsic symplectic symmetric spaces and
  exhibit some of their properties. We construct large families of
  examples and show how they fit in the perspective of a complete
  classification of these manifolds. We also build a natural
  $\star$-quantization on a class of examples.
\end{abstract}

\tableofcontents

\newpage

\section*{Introduction}\label{section:intro}
From its very early days differential geometry has had an intrinsic and
an extrinsic viewpoint. The theory of surfaces in Euclidean $3$-space
was originally devoted to the description of properties of some subsets 
of the Euclidean $3$-space, before Gauss discovered that certain properties depended
only on the metric induced by the ambient space on the surface.
This opened the way to study abstract objects of dimension $2$ or higher
endowed with a metric or some other structures. The consideration of
actions of symmetry groups led naturally to the notion of symmetric spaces
which can be characterized by their automorphism group, independently
of any embedding in some large Euclidean space.\\
Symmetric symplectic spaces have been studied from different 
viewpoints, but they have not been completely classified (contrary
to Riemannian symmetric spaces or Lorentz type pseudo-Riemannian
 symmetric spaces) and such 
a classification seems remote.  
What is known is the classification of symplectic
 symmetric spaces with completely reducible holonomy \cite{Biel}, the
classification of those spaces with Ricci-type curvature \cite{CGR}
 and the classification of symmetric spaces of dimension $2$ and $4$
 \cite{Biel}. \\
 It seems worthwhile to introduce the extrinsic
point of view and to try and determine those symmetric spaces which are
nice symmetric submanifolds of a symplectic vector space.

In this paper, we define the extrinsic symmetric subspaces of a symplectic
vector space and we give large families of examples; we also prove that
this class of examples exhaust all possibilities in codimension $2$.
It also turns out that an easy construction of quantization is possible.

We stress the fact that extrinsic symmetric spaces have been defined by
 \name{D.~Ferus}  in a Riemannian context  \cite{Riem} and that the reading 
 of his paper was certainly an incentive for our work.
More recently, extrinsic symmetric spaces have been studied in a pseudo-Riemannian
context \cite{pseudoR} and in a CR-context \cite{CR}.

 In section~\ref{sec:basicdefinitions}, we give the basic definitions
 and elementary properties of extrinsic symmetric symplectic spaces. We
 also give an algebraic characterization of those spaces. In
 section~\ref{sec:familyofexamples} we construct a family of
 submanifolds of codimension $2p$ of $(\RR^{2(n+p)},\Omega)$ and prove
 that they are indeed extrinsic symmetric symplectic spaces. In
 section~\ref{sec:algebraicequations} we give a class of solutions of
 the algebraic equations and show how they correspond to the class of
 examples built in section \ref{sec:familyofexamples}. In
 section~\ref{sec:codim2} we describe the extrinsic
 symmetric symplectic spaces of dimension $2n$ embedded in
 $\RR^{2n+2}$. Finally in section~\ref{sec:starquantization} we show
 how to build an explicit $\star$-quantization on a class of  examples.

 \begin{Thanks}
   It is a pleasure for the first two authors to thank their friend
  \name{John~Rawnsley}  for
  stimulating discussions.
 This research has been partially supported by the Schwerpunktprogramm Differentialgeometrie of the Deutsche Forschungsgesellschaft.  
 \end{Thanks}

\section{Definitions and elementary properties}
\label{sec:basicdefinitions}

A symplectic manifold $(\var M,\omega)$ is called symmetric \cite{BCG}
if it is endowed with  a
smooth map
 $$s:\var M\times \var M \rightarrow \var M: (x,y)\rightarrow
s(x,y)=: s_x(y)
$$ such that,
for all $x\in \var M$, $s_x$ is an involutive symplectic diffeomorphism
of $(\var M,\omega)$, called the symmetry at $x$,
for which  $x$ is an isolated fixed point,
and such that for all $x,y$ in $\var M$ one has $s_x s_y s_x=s_{s_x(y)}$.\\
On a symplectic symmetric space $(\var M,\omega,s)$, there is a unique 
symplectic connection $\nabla$, called the symmetric connection,
 for which each symmetry $s_x$
is an affine transformation (a symplectic connection being a torsion free
connection for which the symplectic form is parallel). 
The curvature of this connection is parallel 
 so that $(M,\nabla)$ is an affine symmetric space.

An extrinsic symmetric space is a  manifold whose symmetric structure
 comes from an embedding in a larger  manifold. Let us now make this 
  precise in our symplectic context.

Let $(\var M,\omega)$ and  $(\var P,\nu)$ be two smooth symplectic
manifolds of dimension $2n$ and $2(n+p)$ and let $j : \var M \to \var
P$ be a smooth symplectic embedding (i.e. $j^\star\nu = \omega$). With
a slight abuse of notation, if $x\in \var M$, we write:
\begin{equation*}
  \T_x\var P = \T_x\var M \oplus \left(\T_x\var M\right)^\perp
\end{equation*}
where $\perp$ means orthogonal with respect to $\nu_x$. We denote by
$p_x$ (resp. $q_x$) the projection $\T_x\var P \to \T_x\var M$
(resp. $\T_x\var P \to \left(\T_x\var M\right)^\perp$)
relative to this decomposition. \\
We denote by $\NormBundle\var M$ the normal bundle to $\var M$ with fiber
$\NormBundle\var M_x:=(\T_x\var M)^\perp.$

Let $\flat\nabla$ be any symplectic connection on $(\var P,\nu)$
(i.e.$\flat\nabla$ is torsion free and
$\flat\nabla\nu = 0$).  For $X, Y \in \cdv(\var M)$ (= set of smooth
vector fields on $\var M$) and $\xi,\eta$ smooth sections of
$\NormBundle\var M$, let us define:
\begin{equation}
  \begin{split}
    (\nabla_XY)(x) &:= p_x(\flat\nabla_XY)(x)\\
    (\tnabla_X\xi)(x) &:= q_x(\flat\nabla_X\xi)(x)\\
    \alpha_x(X,Y) &:= q_x(\flat\nabla_XY)(x)\\
    A_\xi X(x)&:=  p_x(\flat\nabla_X\xi)(x).
  \end{split}
\end{equation}
Then $\nabla$ defines a symplectic connection on $(\var M,\omega)$
called the symplectic connection induced through the embedding $j$
by $\flat\nabla$, and $\tnabla$
is a connection on $\NormBundle\var M$ preserving the symplectic structure
on the fibers, i.e.
\begin{equation*}
X\nu(\xi,\eta) = \nu(\tnabla_X\xi,\eta) +
\nu(\xi,\tnabla_X\eta)
\end{equation*}
for any smooth sections $\xi, \eta$ of $\NormBundle{\var M}$.
Finally $\alpha$ is a symmetric bilinear form on $\var M$ with values
 in $\NormBundle\var M$, called the second fundamental form of $\var M$
 and $A$ is known as  the shape operator. The sign has been chosen so that
 \begin{equation*}
 \Omega(\alpha(X,Y),\xi)=\Omega(A_\xi X,Y).
 \end{equation*}
The symplectic curvature tensor of the symplectic connection $\nabla$ of
$\var M$ reads:
\begin{equation}\label{curvature}
  \begin{split}
    R^\nabla(X,Y;Z,T) &= \omega(R^\nabla(X,Y)Z,T) =
    \omega((\nabla_X\nabla_Y -
    \nabla_Y\nabla_X - \nabla_{[X,Y]})Z,T)\\
    &=  R^{\flat\nabla}(X,Y;Z,T) + \nu(\alpha(Y,Z),\alpha(X,T)) -
    \nu(\alpha(X,Z),\alpha(Y,T))
  \end{split}
\end{equation}
where $ R^{\flat\nabla}$ is the symplectic curvature of $\flat\nabla$ on $\var
P$. From now on  {\bf { we consider only embedded symplectic submanifolds
of a symplectic vector space, i.e. $(\var P,\nu) =
(\RR^{2(n+p)},\Omega)$ where $\Omega$ is the standard symplectic form
on $\RR^{2(n+p)}$ and  $\flat\nabla$ is the
standard flat symplectic connection on $(\RR^{2(n+p)},\Omega)$}}. In
this case formula \eqref{curvature} becomes:
\begin{equation}
  \label{eq:curvatureinflatspace}
  R^\nabla(X,Y,Z,T) = \nu(\alpha(Y,Z),\alpha(X,T)) -
    \nu(\alpha(X,Z),\alpha(Y,T))
\end{equation}

Let $x \in (\RR^{2(n+p)},\Omega)$ and let $W$ be a
$2n$-dimensional symplectic vector
  subspace of $\RR^{2(n+p)}$.  We define the \Defn{symmetry}
  $S^W_x$ at $x$, relative to $W$, as the affine symplectic
  transformation of $(\RR^{2(n+p)},\Omega)$ given by:
  \begin{equation}
    S^W_x y = S^W_x (x + (y - x)) = x - p^W (y-x) + q^W(y-x)=y-2p^W(y-x)
  \end{equation}
where $p^{W}$
  (resp. $q^{W}$) defines the projection of $\RR^{2(n+p)}$ on $W$ (resp. $W^\perp$)
  relative to the decomposition $\RR^{2(n+p)} = W\oplus
W^\perp$, where $\perp$ means orthogonal with respect to $\Omega$. 
    The symmetry $S^W_x$ is involutive (${(S_x^W)}^2 = \Id$), fixes all
  points of the affine subspace $x + W^\perp$ and induces the usual symmetry of the
  affine subspace $x+W$ relative to $x$.

\begin{definition}
  An embedded $2n$-dimensional symplectic submanifold $(\var M,
  \omega, \,j )$ of $(\RR^{2(n+p)},\Omega)$, where $j:\var M \rightarrow \RR^{2(n+p)}$
  is the embedding, will be called an \Defn{extrinsic
    symplectic symmetric space} if for all $x \in \var M$, the
  symmetry $S^{\T_xM}_x$ stabilizes $\var M$ (i.e.~$S^{\T_x\var M}_x\var
  M\subset\var M$).\\
We identify implicitly $\T_x\var M$ with  the symplectic subspace $j_*T_x\var M$ of 
$(\RR^{2(n+p)},\Omega)$. The symmetry
$S_x^{\T_x\var M}$ will, from now on, simply be denoted by $S_x$; its
restriction to $\var M$ will be written $s_x$; it is smooth since $\var
M$ is an embedded submanifold.
\end{definition}

Let $\nabla$ be the symplectic connection on $(\var M,\omega)$
induced through the embedding $j$ by the standard flat symplectic connection $\flat\nabla$ on
$(\RR^{2(n+p)},\Omega)$. Since $S_x$ is symplectic and maps $M$ into $M$
we have 
\begin{equation*}
  {S_x}_{\ast _y} \T_y\var{M} = \T_{S_xy}\var{M}\qquad\qquad  
  {S_x}_{\ast _y} \T_y\var{M}^\perp = \T_{S_xy}\var{M}^\perp,
\end{equation*}
and since $S_x$ is an affine transformation of $\RR^{2(n+p)}$:
\begin{equation*}
 {S_x}_{\ast}\flat\nabla_XY =\flat \nabla_{{S_x}_{\ast }X}\,{{S_x}_{\ast}Y} \qquad \forall X, Y \in \cdv(\var \RR^{2(n+p)}).
\end{equation*}
Projecting on the tangent space and on the normal space we
have:
\begin{equation}
  \begin{split}
    {s_x}_{\ast _y} \nabla_XY &= \nabla_{{s_x}_{\ast_y }X}{{s_x}_{\ast
      }Y} \qquad \forall X, Y \in \cdv(\var M),\\
    {S_x}_{\ast_y } \alpha_y(X,Y) &=
    \alpha_{{S_x}(y)}\left({S_x}_{\ast_y}X,{S_x}_{\ast_y }Y\right).
  \end{split}
\end{equation}
Hence the proposition (which justifies the terminology):
\begin{proposition}\label{prop:1}Let $(\var M,\omega,\, j)$ be an
  extrinsic symplectic symmetric space. Then for each point $x\in\var
  M$, $s_x$ coincides with the geodesic symmetry at $x$. It is a
  global involutive symplectic diffeomorphism of $\var M$, admitting
  $x$ as isolated fixed point. Hence $(\var M,\omega,s)$ is a symmetric
  symplectic space. The corresponding symmetric connection is
 the induced symplectic
  connection $\nabla$.
   \end{proposition}

The symmetries $S_x$ also stabilize the connection $\tnabla$ in the
normal bundle and hence the second fundamental form is parallel
\begin{equation}
  \label{eq:parallelalpha}
  \tnabla_X \alpha = 0.
\end{equation}

To a symmetric symplectic space $(\var M,\omega, s)$ one
associates \cite{Loos, BCG} an algebraic object called a symmetric triple.
A symmetric triple  is a triple $(\alg G, \sigma, \mu)$ where $\alg G$
is a real Lie algebra, $\sigma$ is an involutive automorphism of $\alg
G$ and $\mu$ is a Chevalley $2$-cocycle of $\alg G$ with values in $\RR$
for the trivial representation of $\alg G$ on $\RR$,
with the following properties
\begin{enumerate}[(a)]
\item $\alg G = \alg K \oplus \alg P$, $\sigma_{\vert \alg K} =
  \Id_{\vert \alg K}$, $\sigma_{\vert\alg P} = -\Id_{\vert \alg P}$
\item The representation $\ad_{\vert \alg P} : \alg K \to \End \alg P$
  is faithful
\item $[\alg P,\alg P]= \alg K$
\item $\forall k \in \alg K$, $\mu(k,\cdot) = 0$ and $\mu_{\vert \alg
    P\times\alg P}$ is symplectic.
\end{enumerate}
The symmetric triple associated to a symmetric symplectic space
is built as follows:
$\alg G$ is the Lie algebra of the transvection group $\var G$ of
$\var M$ (i.e. the subgroup of the affine symplectic group of $(\var
M, \omega,\nabla)$ generated by products of an even number of
symmetries); $\sigma$ is the differential at the identity of the
involutive automorphism of $\var G$ which is the conjugation by
$s_\basep$ (the symmetry at a base point $\basep$ of $\var M$); 
if $\pi : \var G \to \var M : g \mapsto g(\basep)$, 
the Chevalley $2$-cocycle is the pullback by the differential  $\pi_{*e}$
(where $e$ is the neutral element of $G$)
 of the symplectic $2$-form $\omega_{\basep}$. Remark that 
 $\alg K$ is the algebra of the stabilizer $\var K$ of $\basep$ in $\var
G$ and that the
differential $\pi_{\ast e}$ induces a linear isomorphism of $\alg
P$ with $\T_\basep\var M$.

 In the situation we are interested in, the
symmetric space $(\var M,\omega)$ is embedded in a symplectic vector
space $(\RR^{2(n+p)},\Omega)$ and the symmetries $s_x$ of $\var M$ are
induced by affine symplectic transformations  $S_x$ of $\RR^{2(n+p)}$. 
We want to encode this additional information.  
We  assume that $\var M$ is connected.

We define $\var G_1$ to be the subgroup of the  affine symplectic  group of ($\RR^{2(n+p)},\Omega)$, 
$\affinegroup (\RR^{2(n+p)},\Omega):=
\Sp(\RR^{2(n+p)},\Omega) \lsemiprod \RR^{2(n+p)}$, 
generated by the products of an even number of $S_x$ for $x\in M$.\\
Let $X \in \T_\basep \var M$ and let $\Exp_\basep {t X}$ be the
corresponding geodesic of $\var M$. Consider
\begin{equation}
  \label{eq:oneparamsubgroup}
  \psi^X_t = S_{\Exp_\basep \frac{t X}{2}} S_\basep.
\end{equation}
It is a $1$-parametric subgroup of the affine
symplectic group $\affinegroup(\RR^{2(n+p)},\Omega)$ which yields the parallel
transport (both in the tangent and in the normal bundle to $\var M$) along the geodesic of $\var M$
defined by
$s\rightarrow \Exp_\basep sX$ of a parametric distance $t$.
The restriction of $\psi^X_t$ to $\var M$ is a $1$-parametric group of transvections
of $\var M$.  Let $\tilde X$ be the
affine symplectic vector field on $\RR^{2(n+p)}$ associated to
$\psi^X_t$ (i.e. $\tilde X_u = \ddtzero \psi^X_t(u)$); remark that $\tilde X_\basep=X$.
Let $\alg G_1$ be the
subalgebra of the affine symplectic algebra $\affinealgebra (\RR^{2(n+p)},\Omega)
:=\sp(\RR^{2(n+p)},\Omega) \lsemisum \RR^{2(n+p)}$
of $(\RR^{2(n+p)},\Omega)$ generated by the $\tilde X$'s. Then $\var G_1$ is the connected Lie
subgroup of the affine symplectic group  with algebra $\alg G_1$. The group
$\var G_1$ stabilizes $\var M$ and the homomorphism $\var G_1 \to
\operatorname{Diff}\var M : g \mapsto g_{\vert \var M}$ has for image
the transvection group $\var G$ of $\var M$.

We define $\tilde\sigma_1$ to be the involutive automorphism of $\var G_1$ given by 
conjugation by the symmetry $S_\basep$ and $\sigma_1$ to be the automorphism
of $ \alg G_1$ given by its differential.
Remark  that $S_\basep\psi_tS_\basep
= \psi_{-t}$, hence  $\sigma_{1}
(\tilde X) = -\tilde X$.

Consider  $\pi_1 : \var G_1 \to \var M : g_1 \mapsto g_1(\basep)$; let
$\var K_1$ be the stabilizer of $\basep$ in $\var G_1$ and let $\var
  G_1^{\tilde \sigma_1}$ be the group of fixed points of $\tilde
\sigma_1$. Then
\begin{equation*}
  g_1 \in \var G_1^{\tilde \sigma_1} \Rightarrow g_1( \basep) =
  S_\basep g_1 S_\basep (\basep) = S_\basep g_1 (\basep)
\end{equation*}
implies that $g_1(\basep)$ is a fixed point of $s_\basep$, hence
${\var G_{1}^{\tilde\sigma_1}}_{(0)}$, the connected
component of $\var G_1^{\tilde\sigma_1}$, is contained in $\var
K_1$. On the other hand:
\begin{equation*}
  g_1 \in \var K_1 \Rightarrow S_\basep g_1 S_\basep( \basep) = \basep
  \text{ and } {(S_\basep g_1 S_\basep)}_{\ast\basep} = {g_{1}}_{\ast\basep}
\end{equation*}
since ${g_{1}}_{\ast\basep}$ stabilizes $\T_\basep\var M$ and
$\T_\basep\var M^\perp$. Hence $K_1 \subset \var G_1^{\tilde\sigma_1}$
because an affine transformation of $\RR^{2(n+p)}$ is determined by its $1$-jet at any point.\\
Let $\alg K_1$ be the Lie algebra of $\var K_1$; the above shows that
$\alg K_1 = \{ X \in \alg G_1\, \vert \,{\sigma_1}X = X\}$.\\
Define $\alg P_1 = \{ X \in \alg G_1\, \vert \,{\sigma_1} X = -X
\}$). Clearly
$$ \alg G_1 = \alg K_1 + \alg P_1.$$
The map ${\pi_1}_{\ast \Id} : \alg G_1 \to \T_\basep\var M$
induces a linear isomorphism of $\alg P_1$ with $\T_\basep \var M.$
Hence $\alg P_1 = \{ \tilde X \in \alg G_1\, \vert \, X\in T_\basep \var M\,\}.$
Remark that ${\pi_1}_{\ast \Id} (A,a)=A\basep+a$ for $(A,a)\in \alg G_1\subset
\sp(\RR^{2(n+p)},\Omega) \lsemiprod \RR^{2(n+p)}.$ To summarize the above:

\begin{lemma}
To an extrinsic symplectic symmetric space $(\var M,\omega, \, j)$, 
one associates (having chosen a basepoint $\basep$) a 
subalgebra $\alg G_1$ of the symplectic affine algebra 
$\affinealgebra(\RR^{2(n+p)},\Omega)$, stable under 
$\sigma_1$, the conjugation  by $S_{\basep *} =-\Id_{\vert_{T_\basep \var M}}
\oplus \Id_{\vert_{T_\basep \var M^\perp}}$, so that, writing
 $$
 \alg G_1 = \alg K_1 \oplus \alg P_1\quad\mbox{with}\quad\sigma_{1\vert _{\alg K_1} }=
  \Id_{\vert_{ \alg K_1}}\quad \mbox{and}\quad\sigma_{1\vert_{\alg P_1}} = -\Id_{\vert_{ \alg P_1}}
  $$
we have $[\alg P_1,\alg P_1]= \alg K_1$ and there is a bijective linear map
$\lambda: T_\basep \var M \rightarrow \alg P_1$
so that, if  $\lambda(x)=(A(x),a(x)) \in \sp(\RR^{2(n+p)},\Omega) \lsemiprod \RR^{2(n+p)}$ then $A(x)\basep +a(x)=x\quad \forall x\in T_\basep \var M.$ 
\end{lemma}
\begin{definition}Two extrinsic symplectic symmetric spaces $(\var M,\omega,j)$,
  $(\var M^\prime, \omega^\prime, j^\prime)$ of
  $(\RR^{2(n+p)},\Omega)$ will be called \Defn{isomorphic} if there
  exists an element $\Phi$ of the affine symplectic group
  $\affinegroup(\RR^{2(n+p)},\Omega)$ such that $\Phi\circ j = j^\prime$
\end{definition}
Without any loss of generality, 
since we can compose the embedding $j:\var M\rightarrow \RR^{2(n+p)}$
with any symplectic affine transformation of  $\RR^{2(n+p)}$ to get a new
isomorphic extrinsic symmetric space,
we may assume that the embedded symmetric space $(\var M,\omega)$
contains the origin $0$ of $\RR^{2(n+p)}$ and that the tangent
space $\T_0 \var M$ coincides with the symplectic subspace 
$\RR^{2n}$ of $\RR^{2(n+p)}$ spanned by the $2n$ first basis
vectors.  We shall choose $\basep=0$ as basepoint for $\var M$.
Then $S_\basepzero$ identifies with $S_{\basepzero*} $.
We also choose a symplectic basis adapted to the decomposition
$\RR^{2(n+p)}=\RR^{2n}+\RR^{2p}$.
Then
${\pi_1}_{\ast\Id} : \alg G_1\subset \affinealgebra(\RR^{2(n+p)},\Omega)= \sp(n+p)
\lsemisum \RR^{2(n+p)} \to \T_\basep \var M$ reads:
\begin{equation}\label{eq:proj}
  \pi_{1\ast}(Y,y) =y.
\end{equation}
The affine symplectic algebra $\affinealgebra(\RR^{2(n+p)},\Omega)$ has Lie bracket given by
\begin{equation}
  \label{eq:liebracket}
  [(Y,y), (Y^\prime,y^\prime)] = \left([Y,Y^\prime], Yy^\prime -
    Y^\prime y\right).
\end{equation}
An element $(Y,y) \in \alg P_1$ has the property
\begin{equation}
  S_\basepzero (Y,y) S_\basepzero = (S_\basepzero Y S_\basepzero, S_\basepzero y) = -(Y,y)
\end{equation}
hence $ S_\basepzero Y S_\basepzero = -Y$ and $y \in \RR^{2n}.$\\
Thus, there exists a linear map   $\Lambda : \RR^{2n} \to \sp(n+p)$ such
that
 \begin{equation}
  \alg P_1 = \{ (\Lambda(x),x) \telque x \in \RR^{2n}, S_\basepzero
  \Lambda(x) S_\basepzero = - \Lambda(x)\}
\end{equation}
In our chosen basis $\{e_\alpha ; \alpha \leq 2n \}$ of $\RR^{2n}$ and
$\{ f_i ; i \leq 2p \}$ of $\RR^{2p}$, the matrix of $\Omega$ reads
\begin{subequations}\label{eq:linearmap1}
  \begin{equation}\label{eq:matriceoflambda-omega}
    \Omega =
    \begin{pmatrix}
      \omega_\basepzero & 0 \\ 0 & \NhOmega
    \end{pmatrix}
  \end{equation}
  and the matrix of $S_\basepzero$ is
  \begin{equation}
    S_\basepzero =
    \begin{pmatrix}
      -\Id_{2n}&0\\0&\Id_{2p}
    \end{pmatrix}
  \end{equation}
  so that the matrix of $\Lambda(x)$ is of the form
  \begin{equation}
    \Lambda(x) = \begin{pmatrix}
      0&\Lambda_1(x)\\\Lambda_1^\prime(x)&0
    \end{pmatrix}
  \end{equation}
  with
  \begin{equation}
    \transpose{\Lambda_1(x)}\omega_\basepzero + \NhOmega \Lambda_1^\prime (x) = 0
  \end{equation}
\end{subequations}
The algebra $\alg G_1$ is given by $\alg G_1 = \alg P_1 \oplus [\alg P_1,\alg P_1]$.
Since $[\alg P_1, \alg P_1] = \alg K_1$ and since, by equation (\ref{eq:proj}) $\alg K_1\subset
(\sp(n+p),0)$ we have:
\begin{equation}\label{eq:linearmap2}
  \Lambda(x)y - \Lambda(y)x = 0 \qquad x,y \in \RR^{2n}.
\end{equation}
Finally,  $[\alg K_1,\alg P_1] \subset \alg P_1$ implies:
\begin{equation}\label{eq:linearmap3}
  \Lambda([\Lambda(x),\Lambda(y)]z) =
  [[\Lambda(x),\Lambda(y)],\Lambda(z)] \quad x,y,z \in \RR^{2n}
\end{equation}

\begin{lemma}\label{lem:linearmap}
  Let $(\var M,\omega,\, j)$ be a $2n$-dimensional extrinsic symplectic symmetric
  subspace of $(\RR^{2(n+p)},\Omega)$ such that
  $0\in \var M$ and  $\T_0 \var M = \RR^{2n}$.\\
 Then one can associate to $(\var M, \omega, \, j)$ a linear map
  $\Lambda : \RR^{2n} \to \sp(n+p)$ such that
  \begin{enumerate}[(i)]
  \item[(\ref{eq:linearmap1})]$S_\basepzero\Lambda(x) S_\basepzero =
    -\Lambda(x)\quad \forall x \in \RR^{2n}$
  \item[(\ref{eq:linearmap2})]$\Lambda(x)y - \Lambda(y)x = 0 \quad
    \forall x,y\in\RR^{2n}$
  \item[(\ref{eq:linearmap3})]$\Lambda([\Lambda(x),\Lambda(y)]z) =
    [[\Lambda(x),\Lambda(y)],\Lambda(z)] \quad \forall x,y,z \in \RR^{2n}$
  \end{enumerate}
\end{lemma}
Given two isomorphic extrinsic symmetric spaces $(\var M,\omega,\, j)$,
$(\var M^\prime, \omega^\prime,\,  j^\prime)$ one can construct two other
isomorphic spaces $(\var M,\omega,\, j_1)$, $(\var M^\prime,
\omega^\prime, \, j_1^\prime)$ having the additional property that they
both contain the origin and that the tangent space at the origin is
$\RR^{2n}$. Finally the isomorphism $\Phi_1$ between these two spaces
can be modified (by composing if necessary with an element of the
group $\var G_1^\prime$) in such a way that $\Phi_1(0) =0.$
 In the basis of $\RR^{2(n+p)}$ already described, the element
$\Phi_1$ of $\Sp(n+p)$ has matrix of the form
\begin{equation*}
  \Phi_1 =
  \begin{pmatrix}
    A&0\\0&B
  \end{pmatrix}
  \quad A \in \Sp(n), B\in \Sp(p)
\end{equation*}
As $\Phi_1 \circ j_1 = j_1^\prime$ we have $\Exp_{\basep'=0}^\prime \frac{t}{2}
\Phi_1 X = \Phi_1 \Exp_{\basep=0} \frac{t}{2}X$ ($X \in \RR^{2n}$) and  the
group $\var G_1^\prime = \Phi_1 \var G_1 \Phi_1^{-1}$. Hence
\begin{lemma}
  Let $(\var M,\omega,\, j)$, $(\var M^\prime, \omega^\prime,\,  j^\prime)$
  be isomorphic $2n$-dimensional extrinsic symmetric spaces containing
  the origin $\basepzero$ and having $\RR^{2n}$ as tangent space at
  $\basepzero$. Then there exists an element $\Phi_1 \in \Sp(n) \times \Sp(p)$
  such that the associated linear maps $\Lambda\,
  \text{and}\, \Lambda^\prime : \RR^{2n} \to \sp(n+p)$ are related
  by
  \begin{equation}
    \Lambda^\prime(\Phi_1(x)) = \Phi_1 \Lambda(x) \Phi_1^{-1} \quad x
    \in \RR^{2n}
  \end{equation}
\end{lemma}
Conversely given a linear map $\Lambda : \RR^{2n} \to \sp(n+p)$
satisfying properties~(\ref{eq:linearmap1}, \ref{eq:linearmap2},
\ref{eq:linearmap3}) of lemma \ref{lem:linearmap} one can reconstruct
a subalgebra $\alg G_1$ of $\affinealgebra(\RR^{2(n+p)},\Omega)$ by
\begin{equation*}
  \alg G_1 = \alg P_1 \oplus \alg K_1\quad
  \alg P_1 = \{ (\Lambda(x),x) \telque x \in \RR^{2n}\} \quad \alg K_1
  = \{([\Lambda(x),\Lambda(y)],0) \telque x,y \in \RR^{2n}\}.
\end{equation*}
One considers  the connected Lie subgroup $\var G_1$ 
of the affine symplectic group $\affinegroup(\RR^{2(n+p)},\Omega)$
with Lie algebra $\alg G_1$ and the orbit $\var M$ of $\basepzero$ in $ \RR^{2n+2p}$ 
under the action of  $G_1$, 
$\var M = \var G_1( 0)$; the isotropy of $0$ is a
Lie subgroup $\var K_1$ with algebra $\alg K_1$. This orbit is a
symmetric symplectic space and is an extrinsic symmetric space if the
orbit is embedded.

\begin{lemma}
  Given a linear map $\Lambda : \RR^{2n} \to \sp(n+p)$ as in lemma
  \ref{lem:linearmap}, it determines a connected Lie subgroup $\var
  G_1$ of the affine symplectic group. The orbit of the origin under
  $\var G_1$ has a structure of symplectic symmetric space and is an
  extrinsic symmetric symplectic space if the orbit is an embedded
  submanifold.
\end{lemma}

To conclude this paragraph let us relate the linear map $\Lambda$ to
the geometrical properties of $(\var M,\omega)$. Let $x \in \RR^{2n}$
and let $X= (\Lambda(x),x)$ be the corresponding element of $\alg
P_1$. The associated fundamental vector field $X^\ast$ on
$\RR^{2(n+p)}$ is given by
\begin{equation*}
  \begin{split}
    X_u^\ast &= \ddtzero \Exp -t (\Lambda(x),x)\, u\\
    & = \ddtzero \left( \exp{-t
        \Lambda(x)},\frac{\exp{-t\Lambda(x)}-1}{\Lambda(x)}x
    \right)u=\ddtzero \left( \exp{-t
        \Lambda(x)} u+\frac{\exp{-t\Lambda(x)}-1}{\Lambda(x)}x
    \right)\\
    & = - \Lambda(x)u - x
  \end{split}
\end{equation*}
In particular $X^\ast_\basepzero = -x$. The second fundamental form at
$\basepzero$ now reads:
\begin{equation}
  \label{eq:secondfund}
  \alpha_\basepzero(x,x^\prime) = \alpha_\basepzero(X^\ast_\basepzero,X^{\prime\ast}_\basepzero) =
  q_\basepzero \flat\nabla_{(-\Lambda(x)0-x)}(-\Lambda(x^\prime)v-x^\prime) =
  q_\basepzero \Lambda(x^\prime) x = \Lambda(x^\prime) x.
\end{equation}
The symplectic curvature tensor at $\basepzero$ can be expressed using
\eqref{eq:curvatureinflatspace}, \eqref{eq:matriceoflambda-omega} and
\eqref{eq:secondfund}
\begin{equation}
  \label{eq:curvaturewithlambda}
  R_\basepzero(x,y,z,t) = \NhOmega(\Lambda(y)z,\Lambda(x)t) -
  \NhOmega(\Lambda(x)z,\Lambda(y)t) = -\omega([\Lambda(x),\Lambda(y)]z,t)
\end{equation}
It is convenient to express in a slightly different way the map $\Lambda$, the curvature
and the second fundamental form. We use the basis $\{ e_\alpha ; \alpha
\leq 2n ; f_i ; i \leq 2p\}$ of $\RR^{2(n+p)}= \RR^{2n}\oplus \RR^{2p}$
and the associated dual basis  $\{  f^i ; i \leq 2p\}$ of $\RR^{2p}$, so that
$\Omega(f^i,f_j)=\delta^i_j$.  Condition (\ref{eq:linearmap1}) allows us to write
\begin{equation}
  \Lambda(x) = \sum_{i=1}^{2p}( C_i(x)\otimes \underline{ f^i }+f^i\otimes \underline{C_i(x)})
\end{equation}
where $\underline{u} \pardef \Omega(u,\cdot)$
and  $ C_i (x )\pardef  \Lambda(x) f_i $.
Condition (\ref{eq:linearmap2}) reads 
\begin{equation*}
 \Omega(C_ix,y)= \Omega(\Lambda(x)f_i,y) = - \Omega(f_i, \Lambda(x)y) = -
  \Omega(f_i,\Lambda(y)x) =\Omega(C_iy,x)
\end{equation*}
which says that $\forall i \leq 2p$, $C_i \in \sp(n)$. \\
Remark that $ \Omega(C_ix,y)= -\Omega(f_i,\Lambda(x)y)=\Omega(\alpha(x,y),f_i)$;
this gives
\begin{equation}
C_i=A_{f_i} (\basepzero)
\end{equation}
so that the $C_i$ define the shape operator at the base point.
We also have
\begin{equation}
  \label{eq:lambdaviaCI}
  \Lambda(x) = \sum_{i,k=1}^{2p}\NbOmega^{ik}(C_ix \tensor \underline{f_k} + f_k
  \tensor \underline{C_ix})
\end{equation}
where $\sum_{r}\NhOmega_{ir}\NbOmega^{rk} = \delta_i^k$. Hence :
\begin{align}
  \alpha(x,y) &= \sum_{i,k=1}^{2p}\ \NbOmega^{ik} f_k \omega(C_ix,y),\\
  R(x,y)z &=  \sum_{i,j=1}^{2p}\ \NbOmega^{ji} \left( \omega(C_iy,z) C_jx -
    \omega(C_ix,z)C_jy
  \right)\label{eq:curvaturebyCI},\\
  R(x,y) &=  \sum_{i,j=1}^{2p}\  \NbOmega^{ji} \left( C_jx \tensor \underline{C_iy} - C_j y
    \tensor \underline{C_ix} \right).
\end{align}
Condition (\ref{eq:linearmap3}) becomes 
\begin{equation}
  \label{eq:3rdconditionwitCI}
  \begin{split}  \sum_{i,j=1}^{2p}\NbOmega^{ji}
    \left[
      \omega(C_jy,z) C_l C_i x - \omega(C_jx,z) C_lC_iy
    \right] \\
    =\sum_{i,j=1}^{2p} \NbOmega^{ji}
    \left[\omega(C_jy,C_lz)C_ix - \omega(C_jx,C_lz)C_iy
    \right.\\
    + \left.
      \omega(C_jy,C_lx)C_iz - \omega(C_jx,C_ly) C_iz
    \right].
  \end{split}
\end{equation}

\begin{lemma}
  Let $(V,\omega)$ be a symplectic vector space of dimension $2n$; let
  $\mathcal R \subset \Lambda^2V^\star \tensor S^2 V^\star$  be the
  space of algebraic symplectic curvature tensors on $(V,\omega)$:
$$\mathcal{R} = \{ R \in \Lambda^2V^\star \tensor S^2V^\star
  \telque \csum XYZ R(X,Y,Z,T) = 0 \quad \forall X,Y,Z,T \in V\}.$$
  Let
  $\alg L = \sp(V,\Omega)\simeq \sp(n,\RR)$. Then the linear map $\varphi : \Lambda^2\alg
  L \to \mathcal R$ defined by
  \begin{multline}\label{eq:isomwithcurvspace}
    \varphi(A\wedge B)(X,Y,Z,T) = \omega(BY,Z)\omega(AX,T) -
    \omega(AY,Z)\omega(BX,T)\\
    - \omega(BX,Z)\omega(AY,T) + \omega(AX,Z)\omega(BY,T)
  \end{multline}
  is a $\Sp(n,\RR)$ equivariant linear isomorphism. Furthermore the
  map
  \begin{equation*}
    \psi : \Lambda^2\alg L \to \alg L : A\wedge B \mapsto [A,B]
  \end{equation*}
  is also $\Sp(n)$-equivariant. If $n \geq 2$, the space $\mathcal R$
   is the sum of two irreducible
  $\Sp(n,\RR)$ modules, $\mathcal R =
  \mathcal{E} \oplus \mathcal{W}$, where $\mathcal{E}$ is the space of curvature
  tensors of Ricci type. Then $\ker \psi = \varphi^{-1}(\mathcal{W})$.
\end{lemma}
\begin{proof}
  The map $\varphi$ is well defined; indeed if
  \begin{equation*}
    A^\prime = a A + b B \qquad B^\prime  = c A + d B \qquad ad - bc
    = 1
  \end{equation*}
  $\varphi(A\wedge B) = \varphi(A^\prime\wedge B^\prime)$; also
  $\varphi(A\wedge B) = - \varphi(B \wedge A)$. One checks that
  $\varphi(A\wedge B)$ belongs to $\mathcal{R}$ and
  \begin{equation*}
    \begin{split}
      \varphi(SAS^{-1}\wedge SBS^{-1})(X,Y,Z,T) &= \varphi(A\wedge
      B)(S^{-1}X,S^{-1}Y,S^{-1}Z,S^{-1}T)\\ &= (S\varphi(A\wedge
      B))(X,Y,Z,T)
    \end{split}
  \end{equation*}
  Furthermore
  \begin{equation*}
    \text{ric}_{\varphi(A\wedge B)} (X,Y) = \Tr[Z\mapsto \varphi(A\wedge
    B)(X,Z)Y] = - \omega([A,B]X,Y)
  \end{equation*}
  This implies that $\varphi(\Lambda^2\alg L) \supset \mathcal{E}$;
  also if $n \geq 2$ one constructs an example of $A, B \in
  \sp(n,\RR)$ such that $\varphi(A\wedge B) \neq 0$ and $[A,B] =
  0$. Hence $\varphi(\Lambda^2\alg L) \supset \mathcal{W}$ and
  $\varphi$ is surjective. Equality of dimension implies that it is a
  linear isomorphism.
\end{proof}
From \eqref{eq:curvaturebyCI} and \eqref{eq:isomwithcurvspace} one
sees that
\begin{equation}
  \label{eq:curvaturebyCI-wedged}
  \varphi^{-1}R = \frac12 \Omega^{ji} C_i \wedge C_j
\end{equation}
The left hand side of (\ref{eq:3rdconditionwitCI}) is $- C_l R(x,y)z$; the two first terms of the
right hand side are $- R(x,y)C_lz$; the last two terms depend only on
$\varphi^{-1}R$. Hence we have
\begin{lemma}
  The relation \eqref{eq:linearmap3} of lemma \ref{lem:linearmap} is a
  condition depending only on curvature.
\end{lemma}


\section{A family of examples}
\label{sec:familyofexamples}
We consider a symplectic submanifold of $(\RR^{2(n+p)},\Omega)$ which
is the set of common zeros of $2p$ polynomials of degree $2$. More
precisely, let $\{X_i = (A_i,a_i) ; i \leq 2p \}$ be a $2p$-dimensional
subalgebra of the affine symplectic algebra. If $X_i^\star$ 
are the associated fundamental vector fields
on $\RR^{2p}$ (i.e. $X_i^\star(x)= \ddtzero
\Exp(-tX_i)\cdot x = - (A_ix+a_i)$), the corresponding hamiltonians $F_i$
(functions so that $\inner(X_i^\star)\Omega = d F_i$) can be chosen to be :
\begin{equation}
  \label{eq:hamilfun}
  F_i(x) = \frac12 \Omega(x,A_ix) - \Omega(a_i,x).
\end{equation}
The set $\var\Sigma = \{ x \telque F_i(x) = 0 ; i \leq 2p\}$ is a
$2n$-dimensional embedded symplectic submanifold if $\forall x \in
\var\Sigma$, the subspace of $\T_x\RR^{2(n+p)}$ spanned by the $X_i^\star(x)$
is a $2p$-dimensional symplectic subspace. Indeed, in this case the
map 
$$\Phi : \RR^{2(n+p)} \to \RR^{2p} : x \mapsto (F_1(x), \ldots,
F_{2p}(x))
$$
 is a submersion at all $x\in\var\Sigma$ because
\begin{equation*}
  \Phi_{\ast x}(X_j^\star) =\sum_{k=1}^{2p} \Omega(X_k^\star(x), X_j^\star(x)) \partial_k.
\end{equation*}
The condition that the matrix $ \Omega(X_k^\star(x), X_j^\star(x))$ has rank $2p$ reads
\begin{eqnarray*}
  \Omega(X_j^\star(x), X_k^\star(x)) &=& \Omega(A_jx+a_j, A_kx+a_k) \\
  &= &\frac12
  \Omega(x,[A_k,A_j]x) -\Omega(A_ka_j-A_ja_k,x) +\Omega(a_j,a_k)
\end{eqnarray*}
and since  the $X_i$ span an algebra
\begin{equation*}
  [X_i,X_j] = ([A_i,A_j], A_i a_j - A_j a_i) = \sum_kc_{ij}^k (A_k,a_k)
\end{equation*}
we have, when $x \in \var\Sigma$, 
\begin{equation}
  \label{eq:example:omegaisconstant}
 \Omega(X_j^\star(x), X_k^\star(x)) = \sum_{l}c_{kj}^l F_l(x) +
 \Omega(a_j,a_k) = \Omega(a_j,a_k).
\end{equation}
The assumption is thus that the $2p\times 2p$ matrix $\NhOmega_{jk}:=\Omega(a_j,a_k)$ has rank
$2p$. The normal space to $\var\Sigma$ at $x$, $\NormBundle{}_x\var\Sigma$ is
spanned by the $X_i^\star(x)$; in particular the normal space at the
origin is spanned by $\{a_i ; i \leq 2p\}$.
The connection induced on $\var\Sigma{}$ by the flat connection
$\flat\nabla$ on $\RR^{2(n+p)}$ is :
\begin{equation}
  \label{eq:connectioninexample}
  \begin{split}
    \nabla_XY &= \flat\nabla_XY - \sum_{i,j}\NbOmega^{ij}
    \Omega(\flat\nabla_XY,X_i^\star) X_j^\star\\
    &= \flat\nabla_XY - \sum_{i,j}\NbOmega^{ij} \Omega(Y,A_iX) X_j^\star
  \end{split}
\end{equation}
if $X,Y \in \cdv(\var\Sigma)$ and 
$\NbOmega^{ik}\NhOmega_{kj} = \delta^i_j$.
Its curvature at $x$ is given by :
\begin{equation}
  \label{eq:curvatureinexample}
  R_x(X,Y) = \sum_{i,j}\NbOmega^{ij} (p_x A_jY\tensor \underline{p_xA_iX} + p_x
  A_iX \tensor \underline{p_x A_jY})
\end{equation}
where  $p_x$ is the symplectic orthogonal projection on $\T_x\var\Sigma$ and
$\underline{u} = \Omega(u,\cdot)$. The expression of the covariant
derivative of the curvature involves the covariant derivative of the
following endomorphisms
\begin{equation*}
  \nabla_T(A_i (\cdot)- \sum_{j,k}\NbOmega^{kj}
  \Omega(A_i(\cdot),X_k^\star)X_j^\star)_{\vert x} =\sum_{j,k} \NbOmega^{kj} (p_x A_jT\tensor \underline{p_xA_iX_k^\star} + p_x
  A_iX_k^\star \tensor \underline{p_x A_jT}).
\end{equation*}
It vanishes identically if $\forall\,  i, k, (A_i X_k^\star)(x)$ belongs
to $\NormBundle_x\var\Sigma$. This is the case if there exist some
constants $B_{ij}^k$ such that:
\begin{equation}
  \label{eq:example-stabilized}
  A_i X_j^\star = \sum_k B_{ij}^k X_k^\star
\end{equation}
which yield
\begin{subequations}
  \begin{equation}
    A_iA_j = \sum_k B_{ij}^k    A_k 
      \end{equation}
        \begin{equation}
 A_i a_j=
    \sum_k B_{ij}^k a_k.
      \end{equation}
\end{subequations}
Remark that the $A_i \in \sp(n+p)$ stabilize $\NormBundle{}_0\var\Sigma$ which is the space spanned by $\{\,a_i ; i \leq 2p\,\}$, and thus stabilize also 
 $(\NormBundle{}_0\var\Sigma)^\perp=T_0\var\Sigma$. We denote
$A_{i \vert_{\NormBundle{}_0\var\Sigma}}=:B_i$ and  $A_{i\vert_{T_0\var\Sigma}}=:C_i$.

We now prove that the $2n$-dimensional submanifolds $\var\Sigma{}$
constructed above are extrinsic symmetric spaces by
showing that for any $x, y \in \var\Sigma$, $S_xy$ belongs to
$\var\Sigma$:
\begin{equation*}
  \begin{split}
    F_i(S_xy) &= \frac 12\, \Omega(S_xy, A_i S_xy) - \Omega(a_i, S_x
    y)\\
    &= \frac12\,\Omega(y - 2 p_x(y-x), A_i(y - 2 p_x(y-x))) -
    \Omega(a_i,
    y-2 p_x(y-x))\\
    &= \frac12\,\Omega(y, A_iy) - \Omega(a_i, y) - \Omega(p_x(y-x),
    A_iy) - \Omega(y, A_i p_x(y-x))\\
    & \qquad +2\,\Omega(p_x(y-x),A_ip_x(y-x)) + 2\,\Omega(a_i, p_x(y-x))\\
    &= F_i(y) - 2\,\Omega(p_x(y-x), A_i y) + 2\,\Omega(p_x(y-x),A_i(p_x(y-x)+ q_x(y-x)) \\
    &\qquad+ 2\,\Omega(a_i, p_x(y-x))- 2\,\Omega(p_x(y-x),A_i q_x(y-x)) \\
    &= F_i(y) - 2\,\Omega(p_x(y-x), A_i y-A_i(y-x)+a_i) )\\
    &= F_i(y) + 2\,\Omega(p_x(y-x), -A_i x - a_i) = F_i(y).
  \end{split}
  \end{equation*}
  since $A_iq_x (\cdot)$ is in $\NormBundle{}_x\var\Sigma$ so that $\Omega(p_x(\cdot),A_iq_x (\cdot))=0.$
    Hence:
\begin{theorem}
  Let $\{ X_i = (A_i,a_i) ; i \leq 2p\}$ be $2p$ 
elements of the affine symplectic algebra
$\affinealgebra(\RR^{2(n+p)},\Omega)$ such that
\begin{enumerate}
\item $\{a_i ; i  \leq 2p\}$ span a $2p$-dimensional symplectic
  subspace of $\RR^{2(n+p)}$
\item \label{eq:thm:cond2} there exist constants $B_{ij}^k$ so that
$$\forall i, j, \quad A_i A_j = \sum_k B_{ij}^kA_k \quad \mbox{and}\quad  A_i a_j = \sum_k B_{ij}^ka_k.$$
\end{enumerate}
Then 
\begin{equation*}
 \var \Sigma := \{ \,x \in \RR^{2(n+p)} \telque F_i(x) := \frac12
  \Omega(x,A_ix) - \Omega(a_i,x) = 0\quad \forall i\le 2p\,\}
\end{equation*}
 is a $2n$-dimensional extrinsic symmetric symplectic space in $\RR^{2(n+p)}$.
This space  has zero curvature if and only
if the element $\sum_{ij}\NbOmega^{ij} C_i\wedge C_j$ of $\Lambda^2(\sp(n))$
vanishes, where $\NhOmega_{ij} := \Omega(a_i,a_j),\quad \sum_j\NbOmega^{ij}\NhOmega_{jk} = \delta_k^i$
and where $C_i=A_{i\vert_{T_0\var\Sigma}}.$
\end{theorem}
We shall now exhibit a few properties of the $A_i$'s.
We have
\begin{equation*}
  A_iA_j + A_jA_i = \sum_k(B_{ij}^k + B_{ji}^k) A_k.
\end{equation*}
The left hand side is antisymplectic and the right hand side is
symplectic, hence they both vanish
\begin{equation}
  \label{eq:example:Aianticommute}
  A_iA_j + A_jA_i = 0 \qquad \sum_k(B_{ij}^k + B_{ji}^k)
  A_k = 0.
\end{equation}
This shows in particular that 
\begin{equation}\label{triplezero}
A_iA_jA_k=0  \qquad \forall i,j,k.
\end{equation}

The $B_i:=A_{i\vert{\NormBundle_0\var\Sigma}}$ satisfy $B_iB_j=\sum_k B_{ij}^kB_k$ hence
$B(u)B(v)=B(B(u)v)$ if $B(u)=\sum_{i}u^iB_i$ for $u=\sum_iu^i a_i.$
On the other hand, each $B_i$ is in $\sp(\NormBundle_0\var\Sigma,\NhOmega)$:
\begin{equation}
  \Omega(B_ia_j,a_k) =  \Omega(A_ia_j,a_k)= - \Omega(a_j,A_ia_k)=- \Omega(a_j,B_ia_k).
\end{equation}
Hence there is
an associative structure on $\RR^{2p}:=\NormBundle_0\var\Sigma$ (which is the space spanned by $\{\,a_1,\ldots ,a_{2p}\,\}$) defined by
$$
u \bullet v:=B(u)v
$$
so that $\NhOmega(B(u)v,w)+\NhOmega(v,B(u)w)=0\quad \forall u,v,w.$\\
An example of such an associative structure is given as follows. If we choose a basis 
$\{\,g_i\quad  i\le 2p\,\}$ of
$\NormBundle_0\var\Sigma$ relative to which the matrix of $\NhOmega$ reads $ \begin{pmatrix}
    0&\Id_p \\ -\Id_p & 0
  \end{pmatrix}$
we can define  $u\bullet v=B(u)v$ where 
$$B(g_k)=0 \quad  \mbox{and} \quad B(g_{p+k})=\begin{pmatrix}
      0&D_k\\0&0
    \end{pmatrix}   \quad\mbox{with}\quad D_k= \transpose{D_k}\quad \forall 1\le k\le p.$$\\

 \begin{lemma}\label{lemma:Bzero} If  $A_iA_j=0 \quad \forall i,j$, the symmetric space 
 depends only on the restrictions $C_i:=A_{i\vert_{T_0\var\Sigma}}$ where
$T_0\var\Sigma= (\NormBundle{}_0\var\Sigma)^\perp$ and $\NormBundle{}_0\var\Sigma$ 
 is the space spanned by $\{\,a_i ; i \leq 2p\,\}$.
More precisely, we have:

  Let $B : \RR^{2p} \to\sp(p)$ and $C : \RR^{2p} \to
  \sp(n)$ be maps such that for all $\xi,\eta\in\RR^{2p}$
  \begin{equation*}
    B(B(\xi)\eta) = B(\xi)B(\eta) \quad\text{and}\quad C(B(\xi)\eta) =
    C(\xi)C(\eta) = 0.
  \end{equation*}
  Let $\var M$ and $\var N$ be the sets
  \begin{equation*}
    \begin{split}
      \var M &= \left\{ (x,u) \in \RR^{2n}\times\RR^{2p}
        \,\middle\vert\, \frac12\,\Omega(x,C(\xi)x) +
        \frac12\,\Omega(u,B(\xi)u) - \Omega(u,\xi)
        = 0 \;\forall \xi \in\RR^{2p}\right\}\\
      \var N &= \left\{ (x,u) \in \RR^{2n}\times\RR^{2p}
        \,\middle\vert\, \frac12\,\Omega(x,C(\xi)x) - \Omega(u,\xi) =
        0 \;\forall \xi \in\RR^{2p}\right\}.
    \end{split}
  \end{equation*}

  Then $\var M = \var N$.
\end{lemma}
\begin{proof}
  We first observe, as in \eqref{triplezero}, that $B(\xi)B(\eta)B(\zeta) = 0$ for all $\xi,\eta,
  \zeta$. Indeed, since
  \begin{equation*}
    B(\xi)B(\eta) + B(\eta)B(\xi) = B(B(\xi)\eta + B(\eta)\xi)
  \end{equation*}
  the right-hand side is both symplectic and antisymplectic, and thus
  vanishes. This shows that the $B$'s anticommute. But then we have
  \begin{multline*}
      B(\xi)B(\eta)B(\zeta) = -B(\eta)B(\xi)B(\zeta) = B(\eta)B(\zeta)B(\xi)\\
      = B(B(\eta)\zeta)B(\xi) = - B(\xi) B(B(\eta)\zeta) = -
      B(\xi)B(\eta)B(\zeta)
  \end{multline*}
  hence the result.

  Now if $(x,u)$ is in $\var M$, then for all $\xi, \eta$ we have
  \begin{equation*}
    \begin{split}
      \Omega(B(\xi)u,\eta) &= - \Omega(u,B(\xi)\eta)\\
      &= - \frac12\,\Omega(x,C(B(\xi)\eta)x) -
      \frac12\,\Omega(u,B(B(\xi)\eta)u)\\
      &= - \frac12\,\Omega(u,B(\xi)B(\eta)u)\\
      &= \frac12\,\Omega(B(\xi)u,B(\eta)u).
    \end{split}
  \end{equation*}
  Replacing $\eta$ with $B(\eta)u$, the above calculation gives
  \begin{equation*}
    \Omega(B(\xi)u,B(\eta)u) = \frac12\,\Omega(B(\xi)u,B(B(\eta)u)u)
  \end{equation*}
  which is zero because $B(\xi)B(\eta)B(u)$ is zero. But this was the
  last term of the previous equation which holds for all $\eta$, hence
  $B(\xi)u = 0$. In particular this means that $(x,u)$ is in $\var N$.

  Conversely, if $(x,u)$ is in $\var N$, we do the same calculation
  and get immediately
  \begin{equation*}
    \Omega(B(\xi)u,\eta) = - \Omega(u,B(\xi)\eta) = - \frac12\,\Omega(x,C(B(\xi)\eta)x)  = 0
  \end{equation*}
  hence $B(\xi)u = 0$, thus $(x,u)$ is in $\var M$.
\end{proof}

\begin{lemma}    
 If the $A_i$'s are linearly independent, then necessarily
 $A_iA_j=0\quad \forall i,j$.
 \end{lemma}
 \begin{proof}
If the $A_i$'s are linearly independent, then 
we have
\begin{equation*}
  \label{eq:example:anticommute+linindep-csq}
  B_{ij}^k + B_{ji}^k = 0
\end{equation*}
On the other hand if we write:
\begin{equation}
    B_{ijk} =\sum_l B_{ij}^l\NhOmega_{lk}
\end{equation}
$B_{ijk}$ is symmetric in the last pair of indices (since $B_i\in \Sp(\RR^{2p},\NhOmega)$.
Being
 anti-symmetric in the first pair of indices and symmetric
in the last pair, it is identically zero. 
\end{proof}
If $C_iC_j=0\quad \forall i,j$, then  
$Im:= \oplus^{2p}_{i=1}\Im(C_i)$ is included in $K:=\cap^{2p}_{i=1}\Ker(C_i)$
and $\Omega(Im, K)=0$. There is thus a Lagrangian subspace containing $Im$ and contained in
$K$. In an adapted basis, we have
\begin{equation}
   C_i =\begin{pmatrix}
      0&{\tilde{C}}_i\\0&0
    \end{pmatrix} \text{ with }{\tilde{C}}_i= \transpose{{\tilde{C}}_i}, \forall i \leq 2p.
\end{equation}
This is realizable with linearly independent $C_i$'s if $2p \leq \frac{n(n+1)}{2}$.\\

When $A_iA_j=0\quad  \forall i,j$,  we may choose $\{ a_i ; i \leq 2p \}$ as
a basis of $\RR^{2p}$ and write an element $x \in \RR^{2(n+p)} =
\RR^{2n} \oplus^\perp \RR^{2p}$, $x = y + u$. The functions defining
$\Sigma$ can be chosen in view of Lemma \ref{lemma:Bzero}  to have the special form:
\begin{equation}
  F_i(x = y+u) = \frac12 \omega(y,C_iy) - \Omega(a_i,u)
\end{equation}
Thus $\var\Sigma$ is a graph of a function $\RR^{2n} \to \RR^{2p}$ and in
particular is diffeomorphic to $\RR^{2n}$.\\

Remark that there exist solutions where the $A_iA_j$'s and even the $C_iC_j$'s are not all zero.
For instance, on $\RR^8=\RR^4\oplus \RR^4$ with the
symplectic structure
\begin{equation*}
  \Omega = \begin{pmatrix}
    0&I_2&0&0\\
    -I_2&0&0&0\\
    0&0&0&I_2\\
    0&0&-I_2&0
  \end{pmatrix}.
\end{equation*}
one can define
\begin{equation*}
  A_1 = 0
  \qquad
  A_2 = \begin{permatrix}
    0 & 0 & 0 & 0 & 0 & 0 & 0 & 0\\
    0 & 0 & 0 & 1 & 0 & 0 & 0 & 0\\
    0 & 0 & 0 & 0 & 0 & 0 & 0 & 0\\
    0 & 0 & 0 & 0 & 0 & 0 & 0 & 0\\\hline
    0 & 0 & 0 & 0 & 0 & 0 & 0 & 0\\
    0 & 0 & 0 & 0 & 0 & 0 & 0 & 0\\
    0 & 0 & 0 & 0 & 0 & 0 & 0 & 0\\
    0 & 0 & 0 & 0 & 0 & 0 & 0 & 0
  \end{permatrix}
\end{equation*}
\begin{equation*}
  A_3 = \begin{permatrix}
    0 & 0 & 0 & 0  & 0 & 0 & 0 & 0\\
    1 & 0 & 0 & 0  & 0 & 0 & 0 & 0\\
    0 & 0 & 0 & -1 & 0 & 0 & 0 & 0\\
    0 & 0 & 0 & 0  & 0 & 0 & 0 & 0\\\hline
    0 & 0 & 0 & 0  & 0 & 0 & 0 & 0\\
    0 & 0 & 0 & 0  & 0 & 0 & 0 & 1\\
    0 & 0 & 0 & 0  & 0 & 0 & 0 & 0\\
    0 & 0 & 0 & 0  & 0 & 0 & 0 & 0
  \end{permatrix}
  \qquad
  A_4 = \begin{permatrix}
    0  & 0 & 0 & 1 & 0 & 0 & 0  & 0 \\
    -1 & 0 & 1 & 0 & 0 & 0 & 0  & 0 \\
    0  & 0 & 0 & 1 & 0 & 0 & 0  & 0 \\
    0  & 0 & 0 & 0 & 0 & 0 & 0  & 0 \\\hline
    0  & 0 & 0 & 0 & 0 & 0 & 0  & -1\\
    0  & 0 & 0 & 0 & 0 & 0 & -1 & 0 \\
    0  & 0 & 0 & 0 & 0 & 0 & 0  & 0 \\
    0  & 0 & 0 & 0 & 0 & 0 & 0  & 0
  \end{permatrix}
\end{equation*}
then  $A_3 A_4 = - A_4 A_3 = A_2$ and all other
products vanish.

\section{The algebraic equations on $\Lambda$}
 \label{sec:algebraicequations}
We have seen that one can associate to a $2n$-dimensional extrinsic
symmetric space $(\var M,\omega)$ embedded symplectically in
$(\RR^{2(n+p)},\Omega)$ in such a way that $0 \in \var M$ and
$\T_0 \var M = \RR^{2n}$, a linear map $\Lambda : \RR^{2n} \to
\sp(n+p)$ such that; $\forall x,y,z \in \RR^{2n}$:
$$
S_\basepzero \Lambda(x) S_\basepzero = -\Lambda(x)\quad (\ref{eq:linearmap1})\qquad
\Lambda(x)y = \Lambda(y)x\quad (\ref{eq:linearmap2})\qquad \mbox{and}$$
$$
 \Lambda([\Lambda(x),\Lambda(y)]z) = [[\Lambda(x),\Lambda(y)],\Lambda(z)]\quad (\ref{eq:linearmap3}).
 $$
Equivalently, one can consider a linear map 
$$
C:\RR^{2p}\rightarrow \sp(n)\quad f_i\mapsto C_i 
$$
where the $f_i (i \leq 2p)$ form a basis of $\RR^{2p}=\T_\basepzero \var M^\perp$,
so that :
\begin{equation*}
 \begin{split}  
 \sum_{i,j=1}^{2p}\NbOmega^{ji}
    \left[
      \omega(C_jy,z) C_l C_i x - \omega(C_jx,z) C_lC_iy
    \right]   \\
    =\sum_{i,j=1}^{2p} \NbOmega^{ji}
    \left[\omega(C_jy,C_lz)C_ix - \omega(C_jx,C_lz)C_iy
    \right. \\
    + \left.
      \omega(C_jy,C_lx)C_iz - \omega(C_jx,C_ly) C_iz
    \right]  
 \end{split} \qquad\qquad \qquad\qquad\qquad(\ref{eq:3rdconditionwitCI})
 \end{equation*}
 with $\NhOmega_{ij} = \Omega(f_i,f_j)$ and $\sum_j\NbOmega^{ij}\NhOmega_{jk} =
\delta^i_k$. \\
Given $\Lambda$, the map $C$ is defined by $C_ix=\Lambda(x)f_i$;
given $C$ the map $\Lambda$ is given by
\begin{equation*}
  \Lambda(x) =\sum_{ik} \NbOmega^{ik} (C_i x \tensor \underline{f_k}  + f_k
  \tensor \underline{C_ix}).  \eqno(\ref{eq:lambdaviaCI})
\end{equation*}
We have also seen that given such a $\Lambda$ or $C$, one can construct a Lie
subgroup $\var G_1$ of the affine symplectic group and that the orbit
$\var G_1(0)$ is an extrinsic symmetric space if it is an embedded
submanifold.

\begin{lemma}\label{lem:zerocurvature3.1}
  The subspace of $\RR^{2p}$ spanned by $\{\alpha_\basepzero(x,y) \telque
  x,y\in\RR^{2n}\}$ is isotropic if and only if the curvature
  $R_\basepzero$ vanishes.
\end{lemma}
\begin{proof}
  If the subspace spanned by the $\alpha(x,y)$ is isotropic $R = 0$
  (using (\ref{eq:curvatureinflatspace})).

Conversely if $R = 0$, using (\ref{eq:curvatureinflatspace}) one has
\begin{equation*}
\begin{split}
  \nu(\alpha(x,t),\alpha(y,z)) =   \nu(\alpha(y,t),\alpha(x,z)) =
  \nu(\alpha(t,y),\alpha(z,x))  \\
   = \nu(\alpha(z,y),\alpha(t,x)) = 
  \nu(\alpha(y,z),\alpha(x,t)) = 0
  \end{split}
\end{equation*}
\end{proof}

\begin{lemma}
  The only flat extrinsic symmetric symplectic spaces are the graph of
  quadratic polynomial functions $\RR^{2n} \to \RR^{2p}$ whose image
  is isotropic.
\end{lemma}
\begin{proof}
  The extrinsic symmetric space is the orbit of the origin under the
  action of the group generated by $\{e^{t(\Lambda(x),x)} \telque x
  \in \RR^{2n} \}$. Since
  $\alpha(x,\cdot) = -\sum_{i,k}\NbOmega^{ki} f_k \tensor \underline{C_ix}$ and
  the image of $\alpha$ is isotropic we have:
  \begin{equation*}
    \begin{split}
      \Lambda(x)^2 &=\sum_{ikrs} \NbOmega^{ik} \NbOmega^{rs} \omega(C_ix,C_rx) f_k
      \tensor \underline{f_s}\\
      \Lambda(x)^3 &= 0
    \end{split}
  \end{equation*}
and hence: 
\begin{equation*}
  \frac{\exp{t\Lambda(x)}-1}{\Lambda(x)}x = tx -\frac{t^2}{2}
 \sum_{ki} \NbOmega^{ki} f_k \Omega(C_ix,x)
\end{equation*}
which proves the lemma.
\end{proof}

\begin{theorem}\label{thm:3.3}
  Let $\{ f_i ; i \leq 2p \}$ be a basis of ($\RR^{2p},\NhOmega)$ and let $\{f^i
  ; i \leq 2p \}$ be the dual basis $(\Omega(f^i, f_j) =
  \delta^i_j).$ Let $C_i,\, i\le 2p$ be $2p$ elements of the symplectic algebra $\sp(n)$
so that 
\begin{equation}
C_i C_j =\sum_k B_{ij}^k C_k 
\end{equation}
for some constants  $B_{ij}^k$ satisfying
\begin{equation}\sum_rB_{ij}^r\NhOmega_{rk}+\sum_rB_{ik}^r\NhOmega_{jr}=0.
\end{equation}
 Then the $C_i$ are a solution
of (\ref{eq:3rdconditionwitCI}); the corresponding group $\var G_1$ is
a nilpotent subgroup of the affine symplectic group
$A_{2(n+p)}$; the orbit of the origin $0\in \RR^{2(n+p)}$ under this group is 
the connected component of $0$ of the
extrinsic symmetric space $\var\Sigma$ defined by :
\begin{eqnarray}
\nonumber\var\Sigma&=& \{\,(x,u)\in\RR^{2(n+p)}=\RR^{2n}\oplus\RR^{2p}\,\vert \\ 
&&\qquad\frac12\Omega(x,C_ix)+\frac12\Omega(u,B_iu)-\Omega(f_i,u)=0\quad 1\le i\le 2p\,\}.
\label{eq:surf}\end{eqnarray}
where $B_i:\RR^{2p}\rightarrow \RR^{2p}$ is the linear symplectic map such that 
$B_i(f_j)=\sum_{k}B_{ij}^{k}f_k.$
\end{theorem}
\begin{proof}
The fact that such $C_i$'s give  a solution of
(\ref{eq:3rdconditionwitCI}) is straightforward. 
Oberve that, as before, all $C_iC_j+C_jC_i$ vanish, so $C_iC_jC_k=0\,\forall i,j,k.$
One has
\begin{eqnarray*}
  \Lambda(v) &=&\sum_{i}( f^i \tensor \underline{C_iv} + C_iv \tensor \underline{f^i}) \cr
  \Lambda(v)   \Lambda(w) &=&   \sum_{ij} (\Omega(C_iv,C_jw)f^i\tensor \underline{f^j}-
  \NbOmega^{ij}C_iv\tensor\underline{C_jw})\cr
  \Lambda(v)\Lambda(w)\Lambda(x)&=&-\sum_{ijk}(\Omega(C_iv,C_jw)\NbOmega^{jk}f^i\tensor \underline{C_kx} \cr
&& \qquad\qquad \qquad
+\Omega(C_jw,C_kx)\NbOmega^{ij}C_iv\tensor\underline{f^k})\cr
  \Lambda(v)\Lambda(w)\Lambda(x)\Lambda(y)&=&  \sum_{ijkr}
  \Omega(C_jw,C_kx)\NbOmega^{ij}\NbOmega^{kr}C_iv\tensor\underline{C_ry})\cr
   \Lambda(v)\Lambda(w)\Lambda(x)\Lambda(y)\Lambda(z)&=&
    \sum_{ijkrs}\Omega(C_jw,C_kx)\NbOmega^{ij}\NbOmega^{kr}
   \Omega(C_ry,C_sz)C_iv\tensor\underline{f^s}\cr
   &=&0
\end{eqnarray*}
because $\sum_{j}\Omega(C_iv,C_jw)\NbOmega^{jk}=-\sum_{jn}B_{ji}^n\NbOmega^{jk}\Omega(C_nv,w)=
-\sum_{jn} B_{ji}^k\NbOmega^{jn}\Omega(C_nv,w)$ and $\sum_{js}B_{ji}^sB_{kl}^jC_s=C_kC_lC_i=0$
hence $\sum_{jk}(\Omega(C_iv,C_jw)\NbOmega^{jk}\Omega(C_kx,C_ry)=0.$\\

The algebra $\alg G_1 = \alg
P_1 + \alg K_1\subset \sp(n+p) \lsemiprod \RR^{2(n+p)}$ is defined by 
\begin{eqnarray*}
    \alg P_1 &=& \{ (\Lambda(x),x) \telque x \in \RR^{2n}\}\cr
       \alg K_1 &= &[\alg P_1,\alg P_1]=  {([\Lambda(x),\Lambda(x')],0) \telque x,x^\prime \in \RR^{2n}\}}.
\end{eqnarray*}
One sees that $\alg K_1$ is a $2$-step nilpotent  algebra and that $\alg
G_1$ is  nilpotent.\\
 Since  $\Lambda{(x)}^5=0$, $\forall x \in
\RR^{2n}$ the orbit of the origin under the action of the group $\Exp
t (\Lambda(x),x)$ is the set of points $(\widetilde{x(t)},\widetilde{u(t)})= \frac{\exp{t\Lambda(x)} - 1}{\Lambda(x)}x  \in
  \RR^{2n} \oplus \RR^{2p}$ with
\begin{eqnarray*}
 \widetilde{x(t)}&=& t x + \frac{t^3}{6} (\Lambda(x))^2x+\frac{t^5}{5!}(\Lambda(x))^4x\cr
 &=& tx- \frac{t^3}{6}\sum_{ij}\NbOmega^{ij} C_ix\Omega(C_jx,x))
 +\frac{t^5}{5!} \sum_{ijkr}\Omega(C_jx,C_kx)\NbOmega^{ij}\NbOmega^{kr}C_iv\Omega(C_rx,x))\cr
 \widetilde{u(t)}&=& \frac{t^2}{2} \Lambda(x)x+\frac{t^4}{4!}(\Lambda(x))^3x\cr
 &=& \sum_if^i  [\frac{t^2}{2}\omega(C_ix,x)
  -\frac{t^4}{4!}\sum_{jk}(\Omega(C_ix,C_jx)\NbOmega^{jk} \Omega(C_kx,x)]\cr
  &=& \sum_if^i  [\frac{t^2}{2}\omega(C_ix,x)
  +\frac{t^4}{4!}\sum_{jkr}B_{ij}^r\NbOmega^{jk}(\Omega(C_rx,x) \Omega(C_kx,x)]
\end{eqnarray*}
Thus $C_k \widetilde{x(t)}=tC_kx- \frac{t^3}{6}\sum_{ij}\NbOmega^{ij} C_kC_ix\Omega(C_jx,x))$ and
\begin{equation*}
\begin{split}
 \Omega(\widetilde{x(t)},C_k\widetilde{x(t)})=t^2\Omega(x,C_kx) - \frac{t^4}{6}\sum_{ij}\NbOmega^{ij} \Omega(x, C_kC_ix)\Omega(C_jx,x)\cr
  - \frac{t^4}{6}\sum_{js}\NbOmega^{js} \Omega(C_jx,C_kx)\Omega(C_sx,x).
 \end{split}
 \end{equation*}
 Also $B_k \widetilde{u(t)}=\frac{t^2}{2}\sum_{ijr}B_{kj}^r\NbOmega^{ij}   \omega(C_ix,x)f_r$ and
\begin{equation*}
\Omega(\widetilde{u(t)},B_k\widetilde{u(t)}) =
  +  \frac{t^4}{4}\sum_{sjr}\omega(C_rx,x)B_{kj}^r\NbOmega^{sj} r \omega(C_sx,x)
\end{equation*}
whereas 
\begin{equation*}
\Omega(\widetilde{u(t)},f_k) =
 \frac{t^2}{2}\omega(C_kx,x)
  +\frac{t^4}{4!}\sum_{jkr}B_{kj}^r\NbOmega^{js}\Omega(C_rx,x) \Omega(C_sx,x)\end{equation*}
so that
\begin{equation*}
  \frac12 \Omega(\widetilde{x(t)},C_k\widetilde{x(t)}) 
  +  \frac12 \Omega(\widetilde{u(t)},B_k\widetilde{u(t)}) = \Omega(f_k,\widetilde{u(t)}).
\end{equation*}
Hence the orbit $\var G.0$ coincides with the connected component of $0$ of the 
surface defined by \eqref{eq:surf}, since this surface is an extrinsic symmetric symplectic space. 
\end{proof}

\section{The situation in codimension $2$}
 \label{sec:codim2}
We prove:
\begin{theorem}
 Let $\{f_1,f_2\}$ be a symplectic basis of $\RR^{2}\, (=\RR^{2p}$ for $p=1)$  and let $\{f^1,
 f^2\}$ be the dual basis. Let $\Lambda : \RR^{2n} \to \sp(n+1)$ be
 defined by
 \begin{equation*}
   \Lambda(x) = \sum_{i=1}^2 f^i \tensor \underline{C_ix} + C_ix
   \tensor \underline{f^i}
 \end{equation*}
where $C_1, C_2$ are elements of $\sp(n)$. Assume the $C_i$'s obey the
relations (\ref{eq:3rdconditionwitCI}). Then either $\Lambda$
corresponds to a flat extrinsic symmetric space, or $C_iC_j = 0$ for
all $i, j$, hence $\Lambda$ corresponds to  extrinsic symmetric
spaces described in theorem \ref{thm:3.3} (for vanishing $B$'s) .
\end{theorem}
\begin{bigproof}
  Let $\hat{\mathcal{C}}$ be the subspace of $\sp(n)$ spanned by the
  elements $C_1$ and $C_2$. If $\dim \hat{\mathcal{C}} \leq 1$, the
  curvature of the space associated to $\Lambda$ is zero by lemma
  \ref{lem:zerocurvature3.1}. So assume from now on that $\dim
  \hat{\mathcal{C}} =2$. The rest of the proof is divided into a sequence of
  lemmas.

  \begin{lemma}
    All elements $C \in \hat{\mathcal{C}}$ are nilpotent.
  \end{lemma}
  \begin{proof}
    Let us complexify $\RR^{2n}$, and denote by $V$ its
    complexification. We may assume that $(V,\omega)$ is a complex
    symplectic vector space. Assume there exists $C \in
    \hat{\mathcal{C}}$ which is not nilpotent; hence $C$ has non-zero
    eigenvalues. Let $\lambda \neq 0$ be an eigenvalue such that
    $\vert\lambda\vert$ is maximal; let $x$ be an eigenvector and let
    $y \in V$ be such that $\omega(x,y) = 1$.

In the case $p=1$ the equations \eqref{eq:3rdconditionwitCI} can be
written
\begin{equation}
  \begin{split}
    \label{eq:proof-eqa}
    &C_1^2 x \rond C_2y - C_1 C_2 x \rond C_1y - C_1^2 y \rond C_2x +
    C_1C_2 y \rond C_1x \\
    &+ \omega((C_1C_2+C_2C_1)y,x)C_1 - 2 \omega(C_1^2y,x)C_2 = 0
  \end{split}
\end{equation}

\begin{equation}
  \begin{split}
    \label{eq:proof-eqb}
    &C_2C_1x \rond C_2 y - C_2^2 x \rond C_1y - C_2C_1 y \rond C_2x +
    C_2^2 y \rond C_1x \\
    &+ 2 \omega(C_2^2y,x) C_1 - \omega((C_2C_1+C_1C_2)y,x)C_2 = 0
  \end{split}
\end{equation}
where
\begin{equation*}
  u \rond v \notation u \tensor \underline{v} + v \tensor \underline{u}
\end{equation*}
We can choose the basis $\{f_1,f_2\}$ of $\RR^2$ such that $C =
C_1$. Now
\begin{equation*}
  \omega(C_1x,y) = \lambda \omega(x,y) = \lambda \neq 0
\end{equation*}
Hence we adjust $f_2$ so that
\begin{equation*}
  \omega(C_2x,y) = 0
\end{equation*}
Then \eqref{eq:proof-eqa} becomes
\begin{equation}
  \begin{split}
    \label{eq:proof-eqc}
    &\lambda^2 x \rond C_2 y - C_1 C_2 x \rond C_1 y - C_1^2 y \rond
    C_2 x + \lambda C_1C_2 y \rond x \\
    &+ \omega(C_2C_1y,x) C_1 + 2 \lambda^2 C_2 = 0
  \end{split}
\end{equation}
Apply this endomorphism to $x$
\begin{subequations}
  \renewcommand{\theequation}{\theparentequation.\roman{equation}}
  \label{eq:proof-eqd}
  \begin{equation}\label{eq:proof-eqda}
    \begin{split}
      &- \lambda C_1 C_2 x + \lambda C_1 y \omega(C_2x,x) - C_1^2 y
      \omega(C_2x,x) + \lambda^2 C_2 x\\
      &+ \omega(C_2C_1y,x)C_1 x + 2\lambda C_2x = 0
    \end{split}
  \end{equation}
  \begin{equation}\label{eq:proof-eqdb}
    \lambda(3\lambda-C_1) C_2 x - \omega(C_2x,x) C_1 (C_1 - \lambda) y +
    \lambda \omega(C_2C_1y,x)x = 0
  \end{equation}
\end{subequations}
Pairing \eqref{eq:proof-eqd} with $x$ gives:
\begin{equation*}
  4 \lambda^2 \omega(C_2x,x) + 2 \lambda^2 \omega(C_2x,x) = 6
  \lambda^2 \omega(C_2x,x) = 0
\end{equation*}
Hence
\begin{equation}
  \label{eq:proof-eqe}
  \omega(C_2x,x) = 0
\end{equation}
Substituting in \eqref{eq:proof-eqd}
\begin{equation}
  \begin{split}
    \label{eq:proof-eqf}
    &\lambda (3 \lambda - C_1) C_2 x - \lambda \omega(C_1C_2x,y) x = 0\\
    &(3\lambda - C_1) C_2 x = \omega(C_1C_2 x ,y) x
  \end{split}
\end{equation}
But $3\lambda$ is not an eigenvalue of $C_1$ by maximality; hence
$3\lambda - C_1$ is invertible; on the other hand
\begin{equation*}
  \begin{split}
    &2 \lambda (3\lambda - C_1)^{-1}x \notation u \Rightarrow (3\lambda
    - C_1) u = 2\lambda x = (3\lambda - C_1)x\\
    &2 \lambda (3\lambda - C_1)^{-1}x = x
  \end{split}
\end{equation*}
Hence
\begin{equation}
  \label{eq:proofeqg}
  2\lambda C_2 x = \omega(C_1C_2x,y) x
\end{equation}
Pairing with $y$ gives
\begin{equation*}
  0 = \omega(C_1C_2 x,y)
\end{equation*}
hence $C_2 x = 0$ by  \eqref{eq:proofeqg}. Substituting in \eqref{eq:proof-eqc} gives
\begin{equation}
  \label{eq:proof-eqh}
  x \rond (\lambda + C_1) C_2 y + 2 \lambda C_2 = 0
\end{equation}
Apply this endomorphism to $y$
\begin{equation}
  \label{eq:proof-eqi}
  \omega((\lambda + C_1) C_2 y,y) x + (3\lambda + C_1 ) C_2 y = 0
\end{equation}
Again $-3\lambda$ is not an eigenvalue of $C_1$ by maximality; i.e
$3\lambda + C_1$ is invertible. As above we get
\begin{equation}
  \label{eq:proof-eqj}
  \begin{split}
    &4 \lambda C_2 y + \omega((\lambda+C_1) C_2y,y) x = 0\\
    &4 \lambda (\lambda + C_1) C_2 y + 2 \lambda \omega((\lambda+C_1)
    C_2 y,y) x = 0
  \end{split}
\end{equation}
  where the second line in \eqref{eq:proof-eqj}  is obtained from the first by miltiplication with $(\lambda+C_1)$.
Pairing with $y$ gives:
\begin{equation*}
  6 \lambda \omega((\lambda+C_1)C_2y,y) = 0
\end{equation*}
Hence by \eqref{eq:proof-eqi}
\begin{equation*}
  (3\lambda + C_1) C_2 y = 0
\end{equation*}
Hence $C_2 y = 0$; using \eqref{eq:proof-eqh} we have $C_2 = 0$ which
contradicts our assumption, $\dim  \hat{\mathcal{C}} = 2$.
  \end{proof}

  \begin{lemma}\label{lemma:bigproof-lemma2}
    Assume that all elements $C \in \hat{\mathcal{C}}$ are nilpotent
    and that equations \eqref{eq:3rdconditionwitCI} are
    satisfied. Assume there exists $C_1 \in \hat{\mathcal{C}}$ such
    that $C_1^2 \neq 0$. Then for all $C \in \hat{\mathcal{C}}$, $\ker
    C_1 \subset \ker C$.
  \end{lemma}
  \begin{proof}
    Let $C_2$ be linearly independent of $C_1$ and $x \in \ker
    C_1$. Then \eqref{eq:proof-eqa} gives:
    \begin{equation}
      \label{eq:proof-eqk}
      - C_1 C_2 x \rond C_1 y - C_1^2 y \rond C_2 x + \omega(C_2C_1
      y,x) C_1 = 0
    \end{equation}
    Apply this to $y$
    \begin{equation*}
      - \omega(C_1y,y) C_1C_2x - 2\omega(C_1C_2x,y) C_1y - \omega(C_2x,y)
      C_1^2 y = 0
    \end{equation*}
    Pairing with $y$:
    \begin{equation*}
      - 3\omega(C_1C_2x,y) \omega(C_1y,y) = 0
    \end{equation*}
    As $C_1 \neq 0$, there exists a dense open set $U$ in $\RR^{2n}$
    such that $\forall y \in U$, $\omega(C_1y,y) \neq 0$; thus
    $\forall y \in U$
    \begin{equation*}
      \omega(C_1C_2x,y) = 0
    \end{equation*}
    Thus $\omega(C_1C_2x,y) = 0\quad \forall y \in \RR^{2n}$; hence:
    \begin{equation*}
      C_1C_2 x = 0
    \end{equation*}
    Substituting in \eqref{eq:proof-eqk} gives:
    \begin{equation*}
      C_1^2 y \rond C_2 x =0
    \end{equation*}
    As $C_1^2 \neq 0$, we can choose $y $ such that $C_1^2y \neq
    0$. Thus
    \begin{equation*}
      C_2x = 0
    \end{equation*}
which proves the lemma.
  \end{proof}

\begin{lemma}\label{lemma:bigproof-lemma3}
  Assume there exists $C_1 \in \hat {\mathcal{C}}$ such that $C_1^2
  \neq 0$ and that relations (\ref{eq:3rdconditionwitCI}) are
  satisfied. Then there is a $0 \neq C_2 \in \hat{\mathcal{C}}$ such
  that $C_2^2 = C_2C_1 = C_1C_2 = 0$.
\end{lemma}
\begin{proof}
  Choose $C_2$ linearly independent of $C_1$ and $x$ such that $C_1^2
  x = 0$, $C_1x \neq 0$. Thus $C_1x \in \ker C_1$ and by the previous
  lemma $C_2C_1 x = 0$. Going back to (\ref{eq:proof-eqa}):
  \begin{equation}
    \label{eq:proof-eql}
    -C_1C_2x \rond C_1y - C_1^2 y \rond C_2 x + C_1C_2 y \rond C_1 x +
    \omega(C_2C_1y,x)C_1 = 0
  \end{equation}
Apply this to $x$:
\begin{equation}
  \label{eq:proof-eqm}
  -\omega(C_1y,x) C_1C_2 x - \omega(C_2x,x) C_1^2 y + \omega(C_1x,x)
  C_1C_2 y + \omega(C_2C_1y,x)C_1x = 0
\end{equation}
Pairing with $x$;
\begin{equation}
  \label{eq:proof-eqn}
  \omega(C_2C_1y,x) \omega(C_1x,x) = 0
\end{equation}
Assume $C_1C_2x \neq 0$; then (\ref{eq:proof-eqn}) implies
$\omega(C_1x,x) = 0$. Pairing (\ref{eq:proof-eqm}) with $y$ gives:
\begin{equation*}
  - 2 \omega(C_1y,x) \omega(C_1C_2x,y) = 0
\end{equation*}
As $C_1 x \neq 0$, there exists an open dense set of $y$'s such that
$\omega(C_1y,x)\neq0$; hence for this open dense set
$\omega(C_1C_2x,y) = 0$, which contradicts our
assumption.  Hence $C_1C_2 x = 0$.

We may thus rewrite (\ref{eq:proof-eql}) as:
\begin{equation}
  \label{eq:proof-eqo}
  - C_1^2 y \rond C_2 x + C_1C_2 y \rond C_1x = 0
\end{equation}
Assume $C_1x$ and $C_2x$ are linearly independent. Then one checks that
there exists $\mu$ such that
\begin{equation*}
  C_1C_2y = \mu(y) C_2x \qquad C_1^2 y = \mu(y) C_1x
\end{equation*}
Hence:
\begin{equation*}
  0  = \omega(C_1^2y,y) = \mu(y) \omega(C_1x,y)
\end{equation*}
As $C_1 x \neq 0$, $\omega(C_1x,y) \neq 0$ on an open dense set; hence
$\mu(y) = 0$ everywhere and $C_1^2y = 0, \forall y$; hence $C_1^2 =
0$ a contradiction. Thus $C_1x$ and $C_2x$ are linearly dependent.

We can thus change basis in $\RR^2$ (replace $C_2$ by $C_2+ \nu C_1$)
in such a way that $C_2x = 0$ (by assumption $C_1x \neq 0$). Going
back to (\ref{eq:proof-eqo}) we have:
\begin{equation*}
  C_1C_2 y \rond C_1 x = 0
\end{equation*}
This implies $C_1C_2 y = 0, \forall y$; hence $C_1C_2 = 0$. But
\begin{equation*}
  0 = \omega(C_1C_2y,z) = \omega(y,C_2C_1z)
\end{equation*}
implies $C_2C_1 = 0$. By $C_1C_2 = 0$ we have $\Im C_2 \subset \ker
C_1$; by lemma \ref{lemma:bigproof-lemma2}, $\ker C_1 \subset \ker C_2$;
hence $C_2^2 = 0$.
\end{proof}

\begin{lemma}\label{lemma:bigproof-lemma4}
  Let $C_1, C_2$ be two linearly independent elements of $\hat
  {\mathcal{C}}$. Assume equation (\ref{eq:3rdconditionwitCI}) is
  satisfied. Then for all $C \in \hat{\mathcal{C}}$, $C^2 = 0$.
\end{lemma}
\begin{proof}
  Assume $C_1^2 \neq 0$; let $C_2$ be an element given by lemma
  \ref{lemma:bigproof-lemma3} such that $C_2^2 = C_2C_1 = C_1C_2 =
  0$. Then relation (\ref{eq:proof-eqa}) reads:
  \begin{equation*}
    C_1^2 x \rond C_2 y - C_1^2 y \rond C_2 x - 2 \omega(C_1^2 y,x)
    C_2 = 0
  \end{equation*}
Apply this to $y$:
\begin{equation*}
  C_1^2  x \omega(C_2y,y) + 3C_2y\omega(C_1^2x,y) - C_1^2 y
  \omega(C_2x,y) = 0
\end{equation*}
Pairing with $y$:
\begin{equation*}
  4 \omega(C_1^2x,y) \omega(C_2y,y) = 0
\end{equation*}
Choose $x$ such that $C_1^2 x \neq 0$. Then an open dense set
$\omega(C_1^2x,y) \neq 0$. Hence $\omega(C_2y,y) = 0,\forall y
$. Hence by polarizing $\omega(C_2y,z) = 0\,\forall y,z$ and $C_2 = 0$
a contradiction.
\end{proof}

We can now complete the proof of the theorem. To do this we first
prove that $\hat{\mathcal{C}}$ must contain an element of rank $\geq
2$.

Let $C_1, C_2$ be a basis of $\hat {\mathcal{C}}$ such that $\rk C_1 =
\rk C_2 = 1$. Then (replacing if necessary $C_1$ and (or) $C_2$ by
$-C_1$ (resp $-C_2$) we have
\begin{equation*}
  C_1 = x_1 \rond x_1 \qquad C_2 = x_2 \rond x_2
\end{equation*}
and $x_1, x_2$ are linearly independent elements of $\RR^{2n}$. But
then:
\begin{equation*}
  C_1 + C_2 = x_1\rond x_1 + x_2 \rond x_2
\end{equation*}
and clearly the image of this operator has dimension $2$.

Le us  thus assume $\rk C_1 \geq 2$. By lemma
\ref{lemma:bigproof-lemma4}, $\forall a,b \in \RR$
\begin{equation*}
  (a C_1 + bC_2)^2 = 0
\end{equation*}
Hence $C_1^2 = C_2^2 = C_1C_2 + C_2 C_1 = 0$. Thus equation
(\ref{eq:proof-eqa}) reads:
\begin{equation*}
  - C_1 C_2 x \rond C_1 y + C_1C_2 y \rond C_1 x = 0
\end{equation*}
Choose $x, y$ such that $C_1 x$ and $C_1y$ are linearly
independent. This is possible as $\rk C_1 \geq 2$. This implies the
existence of $k \in \RR$ such that:
\begin{equation*}
  C_1 C_2 x = k C_1 x  = - C_2 C_1 x
\end{equation*}
Hence $k$ is an eigenvalue of $C_2$; but $C_2$ is nilpotent; thus $k =
0$ and $C_2 C_1 x = 0$. This is valid for all $x$; indeed if $C_1 x =
0$ it is true; and if $C_1 x \neq 0$, one chooses $y$ as above and
come to the conclusion $C_2 C_1 x = 0$. So $C_2 C_1 = 0$.
\end{bigproof}

\section{Quantization of a class of  examples}
\label{sec:starquantization}
Consider the extrinsic symmetric space $\var\Sigma \subset \left(\var
  V=\RR^{2(n+p)} = \RR^{2n} \oplus \RR^{2p}, \Omega\right)$ defined,
as before by
\begin{equation*}
 \var \Sigma = \left\{ z = (x,u) \in V \telque F_i(z)=
    \frac12 \Omega(x,C_ix) - \Omega(a_i,u) = 0 \qquad
    1 \leq i \leq 2p \right\}
\end{equation*}
where $C_i$ for $ 1 \leq i \leq 2p$ are elements in $\sp(\RR^{2n},
\omega_0)$ such that $C_i C_j = 0\quad\forall i,j$, and where $\{a_i
\quad 1 \leq
i \leq 2p\}$ is a basis of $\RR^{2p}$.\\
We denote by $\NhOmega$ the matrix of the restriction of $\Omega$ to
$\RR^{2p}$ in the basis given by the elements $a_k$'s so that
$\NhOmega_{ij}:=\Omega(a_i,a_j)$ and $\NbOmega$ its inverse matrix

The submanifold $\var\Sigma$ is a graph; writing $u =
\sum\limits_{i=1}^{2p} u^i a_i$ we have:
\begin{equation*}
 \var \Sigma = \{ (x,u) \in \RR^{2n}\oplus\RR^{2p} \telque u^j = \frac 12
  \sum_i \NbOmega^{ji} \,\Omega(x,C_ix)\}
\end{equation*}
and we shall use $x \in \RR^{2n}$ as global coordinates on $\Sigma$.

Observe that we have two symplectic transverse foliations on $\var
V$. One defined by the involutive distribution $\NormBundle{}$ spanned
by the vector fields $ (X_{F_i})_{(x,u)} = (- C_ix, -a_i)$ (which
commute) and the other defined by
$\NormBundle{}_{(x,u)}^{\perp_\Omega}$ for which the integral
submanifolds are
\begin{equation*}
  \var \Sigma^c = \{z \in \var V \telque F_i(z) = c_i \quad 1 \leq i \leq 2p\}
  \qquad\text{for $c \in \RR^{2p}.$}
\end{equation*}
We have a natural projection $\pi$ of $\var V$ on $\var\Sigma =\var
\Sigma^0$ along the integral submanifolds of $\NormBundle$. Observe
that the integral curve of $X_{F_i}$ with initial condition $z =
(x,u)$ is given by $z(t) = (x - t C_i(x), u- ta_i)$ and $F_j(z(t)) =
F_j(z) - t \NhOmega_{ij}$.  Hence $\pi$ is given by
\begin{equation*}
  \pi(x,u) = (x - \sum_{i=1}^{2p}(u^i - \frac12 \sum_{j=1}^{2p}
  \NbOmega^{ij} \,\Omega (x,C_jx))\,C_ix\,, \,\frac12 \sum_{j,k} \NbOmega^{kj}\,\Omega(x,C_jx)\,a_k)
\end{equation*}

We identify $C^\infty(\var\Sigma)$ with the functions which are
constant along the leaves of $\NormBundle{}$ i.e.
\begin{equation*}
  C^\infty(\var\Sigma) \simeq \pi^\ast C^\infty(\var\Sigma) = \{ f \in
  C^\infty(\var V) \telque X_{F_i}(f) = 0 \quad\forall 1 \leq i \leq 2p\}
\end{equation*}
The Poisson bracket of two functions on $\var V$ is given by
\begin{equation}
  \{u\, , v\}_{\var V}=\sum_{ij} \Omega^{ij}\partial_iu\partial_jv
\end{equation}
where $\Omega^{ij}$ are the components of  the inverse matrix  $\Omega^{-1}$.\\
Given any point $z = (x,u)$ in $\var V$, we have the splitting $\var V
\simeq \T_{(x,u)}\var V =
\NormBundle{}_{(x,u)}\oplus^{\perp\Omega}\T_{(x,u)}\var\Sigma^c$ for
$c = F(z)$. The Lie derivative of $\Omega$ in the direction of
$X_{F_j}$ vanishes: $\mathcal{L}_{X_{F_i}} \Omega = 0$. Hence
\begin{equation}
  \{ \pi^\ast \tilde f, \pi^\ast \tilde g\}_{\var V} = \pi^\ast
  \{\tilde f, \tilde g\}_{\var\Sigma} ~~ \forall \tilde f, \tilde g \in
  C^\infty(\var\Sigma).
\end{equation}

We consider the Moyal $\star$ product on $(\var V = \RR^{2(n+p)},
\Omega)$.  It is explicitly given by
\begin{equation} \label{eq:WeylMoyal} u\star v = \sum_{r=0}^\infty
  \frac{1}{r!}  \left(\frac{\nu}{2}\right)^r C_r(u, v)
\end{equation}
for $u,v \in C^\infty(\var V)[[\nu]]$, with
\begin{equation}
  C_r(u, v) = 
  \sum_{i's,j's} \Omega^{i_1j_1} \cdots \Omega^{i_rj_r}
  \nabla^r_{i_1\ldots i_r} u \,   \nabla^r_{j_1\ldots j_r} v.
\end{equation}
It is invariant under the affine symplectic group of $(\var V,
\Omega)$. We have
\begin{equation*}
  X_{F_i} (u \star v) = (X_{F_i} u) \star v + u \star (X_{F_i}v)
  \qquad \forall u,v \in C^\infty(\var V)[[v]].
\end{equation*}

Hence
\begin{proposition}
  A deformation quantization is induced on $\var\Sigma$ from the Moyal
  deformation quantization on $\var V$. The star product of two
  functions on $\var\Sigma$ is the restriction to $\var\Sigma$ of the
  Moyal star product of the pullback to $\var V$ of those two
  functions under the projection $\pi$ along the integral submanifolds
  of $\NormBundle$:
  \begin{equation*}
    \tilde f \star_{\var\Sigma} \tilde g = (\pi^\ast \tilde f \star \pi^\ast
    \tilde g)_{\vert \var\Sigma} \qquad \forall \tilde f, \tilde g \in C^\infty(\var\Sigma)[[v]]
  \end{equation*}
  This star product on $\var\Sigma$ is invariant under the action of
  the group of transvections.
\end{proposition}

\end{document}